\documentclass{article}
\usepackage[utf8]{inputenc}

\usepackage{amsfonts}
\usepackage{amssymb}
\usepackage{amsmath}
\usepackage{amsthm}
\usepackage{mathtools}
\usepackage{cleveref}
\usepackage{url}
\usepackage{tikz}
\usepackage{tikz-cd}
\usepackage{booktabs}
\usepackage{float}
\usepackage{afterpage}
\usepackage{subcaption}
\usepackage{appendix}

\usepackage[
backend=biber,
style=alphabetic,
sorting=nyt,
doi=false,
url=false,
isbn=false,
]{biblatex}

\theoremstyle{plain}
    \newtheorem{theorem}{Theorem}
    \newtheorem{proposition}{Proposition}[section]
    \newtheorem{corollary}[proposition]{Corollary}

\theoremstyle{definition}
    \newtheorem{definition}[proposition]{Definition}
    \newtheorem{condition}[proposition]{Condition}
    
    \newtheorem{algorithm}[proposition]{Algorithm}

\theoremstyle{remark}
	\newtheorem{remark}[proposition]{Remark}%

\newcommand{\ZZ}{\mathbb{Z}}
\newcommand{\QQ}{\mathbb{Q}}
\newcommand{\RR}{\mathbb{R}}
\newcommand{\CC}{\mathbb{C}}
\newcommand{\FF}{\mathbb{F}}

\newcommand{\id}{\mathit{id}}

\newcommand{\htpy}{\simeq}

\renewcommand{\epsilon}{\varepsilon}

\DeclareMathOperator{\Hom}{Hom}

\newcommand{\Kh}{\mathit{Kh}}
\newcommand{\BN}{\mathit{BN}}

\newcommand{\ca}{\alpha}
\newcommand{\cb}{\beta}
\renewcommand{\a}{\mathbf{a}}
\renewcommand{\b}{\mathbf{b}}

\title{
    Fixing the functoriality of Khovanov homology: a simple approach
}
\author{Taketo Sano}

\begin{document}

    \maketitle
    
    \begin{abstract}
    Khovanov homology is functorial up to sign with respect to link cobordisms. The sign indeterminacy has been fixed by several authors, by extending the original theory both conceptually and algebraically. In this paper we propose an alternative approach: we stay in the classical setup and fix the functoriality by simply adjusting the signs of the morphisms associated to the Reidemeister moves and the Morse moves.
\end{abstract}

    
    \section{Introduction}\label{sec:intro}

Khovanov \cite{Khovanov:2000} introduced a link homology theory, now known as Khovanov homology, that categorifies the Jones polynomial. He conjectured that the theory is functorial up to sign with respect to isotopy classes of oriented link cobordisms. Jacobsson \cite{Jacobsson:2002} later proved that this is true, with the necessary restriction that the isotopy fixes the boundary links. Bar-Natan \cite{BarNatan:2004} introduced the theory of formal complex of tangle cobordisms, and proved that the functoriality up to sign also holds in his theory. Then Khovanov \cite{Khovanov:2004} proved that the functoriality up to sign holds in the universal Khovanov homology theory, from which the original theory and other deformations can be obtained. 

The sign indeterminacy has been fixed by several authors: Caprau \cite{Cap:2007}, Clack-Morrison-Walker \cite{CMW:2009}, Blanchet \cite{Blanc:2010}, Beliakova et al.\ \cite{Beliakova:2019} and Vogel \cite{Vog:2020}. All of these arguments require extending the original theory, both conceptually and algebraically: using foams or seamed cobordisms in place of simple cobordisms, and more sophisticated algebras that respect the extended setups. Here we propose an alternative approach: we stay in the classical setup, in particular rely on Bar-Natan's theory of formal complexes. 

\begin{theorem}\label{thm:main}
    By some adjustment of signs, the Bar-Natan's functor
    \[
        \BN: \mathit{Cob}^4(\emptyset) \rightarrow \mathit{Kob}(\emptyset),
    \]
    becomes invariant up to chain homotopy under isotopies of link cobordisms (rel boundary).
\end{theorem}

Definition of the functor $\BN$ will be given in \Cref{sec:preliminary}. It is proved in \cite{Khovanov:2004} that any Khovanov-type homology theory can be obtained from the \textit{universal theory}, also called the $U(2)$-equivariant theory, $H_{h, t}(-; \ZZ[h, t])$. $H_{h, t}(-; \ZZ[h, t])$ factors as $H \circ \mathcal{F}_{h, t} \circ \BN$, where $\mathcal{F}_{h, t}$ is the corresponding TQFT and $H$ is the homology functor. Thus we have

\begin{corollary}
    By some adjustment of signs, any Khovanov-type homology theory that can be obtained from a base change of the universal theory becomes strictly functorial with respect to link cobordisms. This includes the original theory \cite{Khovanov:2000}, Lee's theory \cite{Lee:2005}, and Bar-Natan's theory \cite{BarNatan:2004}.
\end{corollary}

The key to proving \Cref{thm:main} is essentially given in our previous paper \cite{Sano:2020}, where we studied the behavior of the canonical classes under Reidemeister moves and cobordisms. As a byproduct, we obtain an isotopy invariant of a closed orientable surface $S \subset \RR^4$ that generalizes the \textit{Khovanov-Jacobsson number} $n_S = |H_{0, 0}(S)| \in \ZZ$ and Tanaka's invariant $|H_{0, t}(S)| \in \ZZ[t]$. 

\begin{theorem} \label{thm:closed-surf-inv}
    Let $H_{h, t}$ denote the universal Khovanov homology theory. For a connected orientable closed surface $S \in \RR^4$, the cobordism map gives an isotopy invariant
    \[
        H_{h, t}(S) \in \Hom(H_{h, t}(\emptyset), H_{h, t}(\emptyset)) = \ZZ[h, t]
    \]
    which equals
    \[
    \begin{cases}
        2 (h^2 + 4t)^\frac{g(S) - 1}{2} & \text{if $g(S)$ is odd,}\\
        0 & \text{if $g(S)$ is even.}
    \end{cases}
    \]
\end{theorem}

The paper is organized as follows. In \Cref{sec:preliminary} we review the basics of Khovanov homology theory in its generalized form, and also Bar-Natan's theory of formal complexes. In \Cref{sec:main} we give the explicit adjustments for the cobordism maps. Finally in \Cref{sec:proof}, we prove the two main theorems, using the results obtained in the previous section. 
    \section{Preliminaries} \label{sec:preliminary}

In this section, we review some basics of Khovanov homology theory, and cite some results from \cite{Sano:2020} that will be needed in the coming sections. Throughout this paper, we will work in the smooth category.

\begin{definition}
    For any unary function $f$, define its \textit{difference function} $\delta f$ by
    \[
        \delta f(x, y) := f(y) - f(x).
    \]
\end{definition}

\begin{definition}
    For an oriented link diagram $D$, let $w(D)$ denote the \textit{writhe} of $D$, $r(D)$ denote the number of Seifert circles of $D$, and $|D|$ denote the number of components of $D$.
\end{definition}

\begin{definition}
    Let $R$ be a commutative ring with unity. A \textit{Frobenius algebra} over $R$ is a quintuple $(A, m, \iota, \Delta, \epsilon)$ such that: 
    \begin{enumerate}
        \item $(A, m, \iota)$ is an associative $R$-algebra with multiplication $m$ and unit $\iota$,
        \item $(A, \Delta, \epsilon)$ is a coassociative $R$-coalgebra with comultiplication $\Delta$ and counit $\epsilon$,
        \item the Frobenius relation holds: 
        \[
            \Delta \circ m = (\id \otimes m) \circ (\Delta \otimes \id) = (m \otimes \id) \circ (\id \otimes \Delta).
        \]
    \end{enumerate}
\end{definition}

\begin{definition}
    For any $h, t \in R$, let $A_{h, t} = R[X]/(X^2 - hX - t)$. The $R$-algebra $A_{h, t}$ is given a Frobenius algebra structure as follows: the counit $\epsilon: A_{h, t} \rightarrow R$ is defined by
    \[
        \epsilon(1) = 0,\quad
        \epsilon(X) = 1.
    \]
    Then the comultiplication $\Delta$ is uniquely determined so that $(A_{h, t}, m, \iota, \Delta, \epsilon)$ becomes a Frobenius algebra. Explicitly, $m$ and $\Delta$ are given by 
    \begin{equation*}
    \begin{gathered}
        m(1 \otimes 1) = 1, \quad 
        m(X \otimes 1) = m(1 \otimes X) = X, \quad
    	m(X \otimes X) = hX + t, \\
        \Delta(1) = X \otimes 1 + 1 \otimes X - h (1 \otimes 1), \quad 
    	\Delta(X) = X \otimes X + t (1 \otimes 1).
    \end{gathered}
    \end{equation*}
\end{definition}

\begin{definition}
    Suppose $h, t \in R$. Given a link diagram $D$ with $n$ crossings, the complex $C_{h, t}(D; R)$ is defined by the construction given in \cite{Khovanov:2000}, except that the defining Frobenius algebra $A = R[X]/(X^2)$ is replaced by $A_{h, t} = R[X]/(X^2 - hX - t)$. Let $H_{h, t}(D; R)$ denote the homology of $C_{h, t}(D; R)$.
\end{definition}

\begin{remark}
    When $R$ is graded (possibly trivially) and $h, t \in R$ are given appropriate degrees, then $C_{h, t}(D; R)$ and $H_{h, t}(D; R)$ admits a direct sum decomposition or a filtration with respect to the \textit{quantum grading}. In this paper, we are only concerned about signs, so we regard $H_{h, t}(D; R)$ merely as an $R$-module.
\end{remark}

$H_{0, 0}(-; \ZZ)$ is Khovanov's original theory \cite{Khovanov:2000}, $H_{0, 1}(-; \QQ)$ is Lee's theory \cite{Lee:2005}, and $H_{h, 0}(-; \FF_2[h])$ is Bar-Natan's theory \cite{BarNatan:2004}. $H_{h, t}(-; \ZZ[h, t])$ is Khovanov's \textit{universal theory} introduced in \cite{Khovanov:2004} (where it is denoted $\mathcal{F}_5$), from which any rank 2 Frobenius algebra based link homology theory can be obtained. 
In contrast to Khovanov's theory, Lee's theory has an amazing property that, for any link diagram $D$, its $\QQ$-Lee homology has dimension $2^{|D|}$ with specific generators, called the \textit{canonical generators}, that are constructed explicitly from $D$. The construction of the canonical generators can be generalized for $C_{h, t}$ provided that the following condition holds.

\begin{condition} \label{cond:ab-cond1}
	The quadratic polynomial $X^2 - hX - t \in R[X]$ factors as $(X - u)(X - v) \in R[X]$ for some $u, v \in R$. 
\end{condition}

Assuming that \Cref{cond:ab-cond1} holds, put $c = v - u$. Obviously $c$ is a square root of $h^2 + 4t$. Note that $c = 0$ for Khovanov's theory, $c = 2$ for Lee's theory, and $c = h$ for Bar-Natan's theory. Define two elements in $A$ by
\[
    \a = X - u,\quad 
    \b = X - v
\]
We call $\a$ and $\b$ \textit{colors}. The multiplication $m$ and comultiplication $\Delta$ of $A_{h, t}$ diagonalizes as: 
\begin{alignat*}{2}
	m(\a \otimes \a) &= c\a, 
	    &\hspace{1cm} 
	     \Delta(\a) &= \a \otimes \a, \\
	m(\a \otimes \b) &= 0, 
	    &\Delta(\b) &= \b \otimes \b \\
	m(\b \otimes \a) &= 0 &&\\
	m(\b \otimes \b) &= -c\b &&
\end{alignat*}
Given an oriented link diagram $D$, we color its Seifert circles by $\a, \b$ according to the following algorithm:

\begin{algorithm} \label{algo:ab-coloring}
    Given a link diagram $D$, color each of its Seifert circles by $\a$ or $\b$ according to the following algorithm: separate $\mathbb{R}^2$ into regions by the Seifert circles of $D$, and color the regions in the checkerboard fashion (with the unbounded region colored white). For each Seifert circle, let it inherit the orientation from $D$, and color $\a$ if it sees a black region to the left, otherwise color $\b$ (see \Cref{fig:ab}).
\end{algorithm}

\begin{figure}[t]
	\centering
    \includegraphics[scale=0.35]{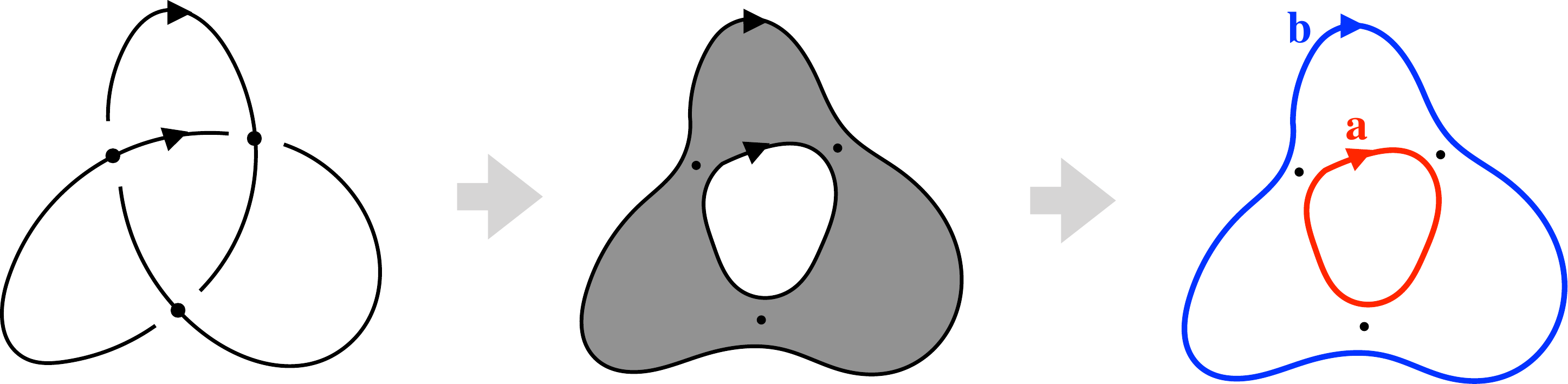}
	\caption{Coloring the Seifert circles by $\a$, $\b$.}
    \label{fig:ab}
\end{figure}

This coloring determines an element $\ca(D) \in C_{h, t}(D; R)$ by the corresponding tensor product of $\a$ and $\b$. On the underlying unoriented diagram of $D$, there are $2^{|D|}$ possible orientations, and for each such orientation $o$, we can apply the same algorithm to obtain an element $\ca(D, o) \in C_{h, t}(D; R)$. It is easily seen that these elements are cycles in $C_{h, t}(D; R)$. These cycles are called the \textit{canonical cycles} of $D$, and those homology classes the \textit{canonical classes} of $D$. In particular, we name the following two cycles for later use. 

\begin{definition}[$\ca$-, $\cb$-cycles]
    Let $o$ be the given orientation of $D$. We define
    \begin{align*}
        \ca(D) &:= \ca(D, o)\\
        \cb(D) &:= \ca(D, -o)
    \end{align*}
    and call them the \textit{$\ca$-cycle, $\cb$-cycle} of $D$. 
\end{definition}

The following proposition is a generalization of \cite[Theorem 4.2]{Lee:2005}, which is proved in \cite[Theorem 4.2]{Turner:2020} and in \cite[Proposition 2.9]{Sano:2020}. 

\begin{proposition} \label{prop:ab-gen}
	If $c = \sqrt{h^2 + 4t}$ is invertible in $R$, then $H_{h, t}(D; R)$ is freely generated over $R$ by the canonical classes. In particular $H_{h, t}(D; R)$ has rank $2^{|D|}$.
\end{proposition}

\begin{remark}
    The universal theory satisfying the condition of \Cref{prop:ab-gen} is the $U(1) \times U(1)$-equivariant theory $H_{\alpha \mathcal{D}}$ \cite{Khovanov:2020}, defined over $R = \ZZ[u, v, c^{-1}]$. 
\end{remark}

\begin{remark}
    The similar statement also holds for Caprau's universal $\mathfrak{sl}(2)$ cohomology over $\CC$ \cite[Theorem 4]{Cap:2009}.
\end{remark}

\begin{remark}
    When $c$ is \textit{not} invertible in $R$, for example for $\ZZ$-Lee theory given by $(R, h, t) = (\ZZ, 0, 1)$, the $\ca$-classes do not necessary generate $H_{h, t}(D; R)$. See \cite{Sano:2020} for details.
\end{remark}

Finally we introduce categories and functors defined by Bar-Natan in \cite{BarNatan:2004} and cite the related theorems.

\begin{definition} The categories $\mathit{Cob}^4(\emptyset)$, $\mathit{Kob}(\emptyset)$ and the functor $\BN$ are defined as follows:
    \begin{itemize}
        \item $\mathit{Cob}^4(\emptyset)$ is the category with objects oriented link diagrams and morphisms 2-dimensional oriented cobordisms between link diagrams generically embedded in $\RR^3 \times [0, 1]$.
        \item $\mathit{Kob}(\emptyset)$ is the category of chain complexes over the additive category $\mathit{Mat}(\mathit{Cob}^3_{/l}(\emptyset))$, where $\mathit{Mat}(\mathit{Cob}^3_{/l}(\emptyset))$ is the category with objects formal direct sums of smoothings of link diagrams and with morphisms formal matrices of cobordisms between such smoothings modulo local relations.
        \item The functor 
            \[
                \BN: \mathit{Cob}^4(\emptyset) \rightarrow \mathit{Kob}(\emptyset)
            \] 
            is defined so that $H_{h, t}(-; R)$ factors as
            \[
                H_{h, t}(-; R): 
                \mathit{Cob}^4(\emptyset) 
                \xrightarrow {\BN}
                \mathit{Kob}(\emptyset) 
                \xrightarrow {\mathcal{F}_{h, t}}
                \mathit{Kom}(R\mathit{Mod})
                \xrightarrow {H}
                R\mathit{Mod}
            \]
            where
            \begin{itemize}
                \item $R\mathit{Mod}$ is the category $R$-modules and $R$-module homomorphisms,
                \item $\mathit{Kom}(R\mathit{Mod})$ is the category of chain complexes over $R\mathit{Mod}$,
                \item $\mathcal{F}_{h, t}$ is the functor (TQFT) corresponding to the Frobenius algebra $A_{h, t}$, and
                \item $H$ is the homology functor.
            \end{itemize}
    \end{itemize}
\end{definition}

A note on the functor $\BN$. For a generic cobordism between link diagrams $S: D \rightarrow D'$ (i.e.\ a morphism in $\mathit{Cob}^4(\emptyset)$), the morphism $\BN(S)$ is defined in the same way as in the original Khovanov's theory. First $\BN(T)$ is defined for each elementary cobordism $T$, which corresponds to a Reidemeister move or a Morse move. Then $\BN(S)$ is defined by decomposing $S$ into elementary cobordisms $T_i$ and composing the corresponding morphisms $\BN(T_i)$. See \cite[Section 8.1]{BarNatan:2004} for the precise definition, where our $\BN$ is denoted $\Kh$. 

\begin{proposition}[{\cite[Theorem 4]{BarNatan:2004}}]
\label{thm:BN}
    The functor $\BN$ is invariant up to chain homotopy and sign under isotopies of link cobordisms (rel boundary).
\end{proposition}

\begin{corollary}
    Let $S: D \rightarrow D'$ be a generic cobordism between link diagrams. For any commutative ring $R$ and $h, t \in R$, the corresponding cobordism map
    \[
        H_{h, t}(S; R) : H_{h, t}(D; R) \rightarrow H_{h, t}(D'; R)
    \]
    is invariant up to sign under isotopies of $S$ (rel boundary).
\end{corollary}

    \section{Adjusting the signs} \label{sec:main}

As reviewed in the previous section, the morphism $\BN(S)$ is defined by composing the morphisms obtained by decomposing $S$ into elementary cobordisms. We adjust the signs of $\BN(S)$ for each elementary cobordism $S$, so that the resulting functor $\BN$ becomes invariant up to chain homotopy, without sign indeterminacy. Note that this adjustment applies to all Khovanov type homologies, since the functor $H_{h,t}(-; R)$ is given by composing $H \circ \mathcal{F}_{h, t}$ with $\BN$. The coming \Cref{lem:adj-rho,lem:adj-phi,lem:cob-no-closed-comp} are mostly proved in \cite[Proposition 2.13, 3.17]{Sano:2020} but we rewrite them for clarity.

\subsection{Adjustments for Reidemeister moves}\label{subsec:adj-reidemeistermoves}

\begin{proposition} \label{lem:adj-rho} \label{prop:variance-under-reidemeister}
    Suppose two link diagrams $D, D'$ are related by one of the Reidemeister moves. Let $S$ be the corresponding elementary cobordism. The sign of $\BN(S)$ can be adjusted so that for any $(R, h, t)$ satisfying $c = \sqrt{h^2 + 4t} \in R^\times$, the corresponding isomorphism 
    \[
        \phi_S: H_{h, t}(D; R) \rightarrow H_{h, t}(D'; R)
    \]
    maps the $\ca$-, $\cb$-classes
    \begin{align*}
        [\ca(D)] &\xmapsto{\phi_S} \hspace{1.5em} c^j [\ca(D')], \\
        [\cb(D)] &\xmapsto{\phi_S} (-c)^j [\cb(D')].
    \end{align*}
    where $j \in \{0, \pm 1\}$ is given by
    \[ 
        j = \frac{\delta w(D, D') - \delta r(D, D')}{2}.
    \]
\end{proposition}

\begin{proof}
    Following \cite[Section 5]{Khovanov:2000}, for any $(R, h, t)$ there is a quasi-isomorphism corresponding to each Reidemeister move 
    \[
        \rho: C_{h, t}(D; R) \rightarrow C_{h, t}(D'; R).
    \]
    In \cite[Proposition 2.13]{Sano:2020} we proved that given $c = \sqrt{h^2 + 4t} \in R^\times$, the induced map $\rho_*$ sends
    \begin{equation}
    \begin{gathered}
        [\ca(D)] \xmapsto{} \epsilon c^j [\ca(D')], \\
        [\cb(D)] \xmapsto{} \epsilon' c^j [\cb(D')].
    \end{gathered}
    \tag{$\star$}
    \label{eq:rho}
    \end{equation}
    Here, $j$ is given as above, and signs $\epsilon, \epsilon' \in \{\pm1\}$ satisfy $\epsilon \epsilon' = (-1)^j$. It is obvious that $\epsilon \rho$ satisfies the properties stated in the proposition. Note that $j$ and the signs $\epsilon, \epsilon'$ are determined solely by $D, D'$, and are independent of $(R, h, t)$. We take one such $(R, h, t)$ satisfying $2 \neq 0 \in R$ (for example $(R, h, t) = (\ZZ, 1, 0)$), and replace $\BN(S)$ with $\epsilon \BN(S)$. Having proved that $\rho_*$ coincides with $\phi_S = H \circ \mathcal{F}_{h, t} \circ \BN(S)$, the proof is done. 
    
    In \cite[Section 4.3]{BarNatan:2004}, the chain homotopy equivalence $F = \BN(S)$ is given explicitly as a linear combination of cobordisms between the formal complexes. For R-I and R-II, one can see that $\mathcal{F}_{h, t} \circ F$ coincides with $\rho$ (the explicit description of $\rho$ is given in \cite[Appendix A]{Sano:2020}). For R-III, $F$ is constructed in a different way and the induced map does not coincide with $\rho$. Nonetheless we can similarly prove the same result for $\phi_S$. Here we briefly outline the proof.
    
    Divide cases by the orientations of the three strands appearing in the move. Fix the orientation of the topmost strand, say, point to the left, and consider the four possible cases for the orientations of the other two strands: $\uparrow\uparrow$, $\downarrow\downarrow$, $\uparrow\downarrow$, $\downarrow\uparrow$ (ordered so that the middle strand comes first). For each case, consider the orientation preserving resolutions for $D$ and $D'$. One sees that for the first three cases, the patterns are isotopic to \Cref{fig:RM3-pattern1}. Using the explicit map $F$ described in \cite[Figure 9]{BarNatan:2004}, one sees that $\eqref{eq:rho}$ holds with $\delta r = 0$. For the forth case (which is the only case where the directions of the strands along the boundary of the disk are alternating), divide into subcases by the way the endpoints are connected outside the disk of the local move (see \Cref{fig:RM3-strands}). For each subcase, the patterns are isotopic to either \Cref{fig:RM3-pattern2} or \Cref{fig:RM3-pattern3} (or its reverse). Again with $F$, one can prove as in \cite[Appendix A]{Sano:2020} that \eqref{eq:rho} holds for each subcase. Note that the coefficient $\pm c^{\pm 1}$ appears only when $\delta r = \mp 2$ (\Cref{fig:RM3-pattern3}).  
\end{proof}

\begin{figure}[p]
    \centering
    \tikzset{every picture/.style={line width=0.75pt}} 

\begin{tikzpicture}[x=0.75pt,y=0.75pt,yscale=-1,xscale=1]

\draw  [dash pattern={on 0.84pt off 2.51pt}] (29,68.3) .. controls (29,49.91) and (43.91,35) .. (62.3,35) .. controls (80.68,35) and (95.59,49.91) .. (95.59,68.3) .. controls (95.59,86.68) and (80.68,101.59) .. (62.3,101.59) .. controls (43.91,101.59) and (29,86.68) .. (29,68.3) -- cycle ;
\draw    (115.57,69.63) -- (152.95,69.63) ;
\draw [shift={(154.95,69.63)}, rotate = 180] [color={rgb, 255:red, 0; green, 0; blue, 0 }  ][line width=0.75]    (10.93,-3.29) .. controls (6.95,-1.4) and (3.31,-0.3) .. (0,0) .. controls (3.31,0.3) and (6.95,1.4) .. (10.93,3.29)   ;
\draw  [dash pattern={on 0.84pt off 2.51pt}] (176.17,68.3) .. controls (176.17,49.91) and (191.08,35) .. (209.47,35) .. controls (227.85,35) and (242.76,49.91) .. (242.76,68.3) .. controls (242.76,86.68) and (227.85,101.59) .. (209.47,101.59) .. controls (191.08,101.59) and (176.17,86.68) .. (176.17,68.3) -- cycle ;
\draw    (49.33,97.62) .. controls (68.35,80.94) and (81.98,64.36) .. (71.27,36.32) ;
\draw  [draw opacity=0][fill={rgb, 255:red, 255; green, 255; blue, 255 }  ,fill opacity=1 ] (54.24,87.82) .. controls (52.7,84.15) and (54.42,79.93) .. (58.08,78.39) .. controls (61.75,76.85) and (65.97,78.57) .. (67.51,82.23) .. controls (69.06,85.9) and (67.34,90.12) .. (63.67,91.66) .. controls (60,93.2) and (55.78,91.48) .. (54.24,87.82) -- cycle ;
\draw  [draw opacity=0][fill={rgb, 255:red, 0; green, 0; blue, 0 }  ,fill opacity=1 ] (58.69,86.18) .. controls (58.2,85) and (58.75,83.65) .. (59.92,83.16) .. controls (61.09,82.67) and (62.44,83.22) .. (62.93,84.39) .. controls (63.42,85.56) and (62.87,86.91) .. (61.7,87.4) .. controls (60.53,87.9) and (59.18,87.35) .. (58.69,86.18) -- cycle ;

\draw    (49.27,37.05) .. controls (44.29,65.03) and (52.19,81.23) .. (75.1,98.43) ;
\draw  [draw opacity=0][fill={rgb, 255:red, 255; green, 255; blue, 255 }  ,fill opacity=1 ] (41.72,62.86) .. controls (38.81,60.15) and (38.64,55.6) .. (41.35,52.69) .. controls (44.06,49.78) and (48.62,49.61) .. (51.53,52.32) .. controls (54.44,55.03) and (54.6,59.59) .. (51.89,62.5) .. controls (49.18,65.41) and (44.63,65.57) .. (41.72,62.86) -- cycle ;
\draw  [draw opacity=0][fill={rgb, 255:red, 0; green, 0; blue, 0 }  ,fill opacity=1 ] (45.1,59.54) .. controls (44.17,58.67) and (44.12,57.22) .. (44.98,56.29) .. controls (45.85,55.36) and (47.3,55.3) .. (48.23,56.17) .. controls (49.17,57.04) and (49.22,58.49) .. (48.35,59.42) .. controls (47.49,60.35) and (46.03,60.4) .. (45.1,59.54) -- cycle ;

\draw  [draw opacity=0][fill={rgb, 255:red, 255; green, 255; blue, 255 }  ,fill opacity=1 ] (70.66,63.11) .. controls (67.75,60.41) and (67.59,55.85) .. (70.3,52.94) .. controls (73.01,50.03) and (77.56,49.86) .. (80.47,52.57) .. controls (83.39,55.28) and (83.55,59.84) .. (80.84,62.75) .. controls (78.13,65.66) and (73.58,65.82) .. (70.66,63.11) -- cycle ;
\draw  [draw opacity=0][fill={rgb, 255:red, 0; green, 0; blue, 0 }  ,fill opacity=1 ] (74.05,59.79) .. controls (73.12,58.92) and (73.06,57.47) .. (73.93,56.54) .. controls (74.8,55.61) and (76.25,55.56) .. (77.18,56.42) .. controls (78.11,57.29) and (78.16,58.74) .. (77.3,59.67) .. controls (76.43,60.6) and (74.98,60.66) .. (74.05,59.79) -- cycle ;

\draw    (29.25,65.61) .. controls (52.76,50.61) and (79.77,53.88) .. (95.87,68.87) ;
\draw    (198.57,98.14) .. controls (196.29,69.77) and (194.87,56.34) .. (219.34,37.24) ;
\draw  [draw opacity=0][fill={rgb, 255:red, 255; green, 255; blue, 255 }  ,fill opacity=1 ] (201.7,51.12) .. controls (200.14,47.46) and (201.85,43.23) .. (205.5,41.68) .. controls (209.16,40.12) and (213.39,41.82) .. (214.95,45.48) .. controls (216.51,49.14) and (214.8,53.37) .. (211.14,54.93) .. controls (207.48,56.48) and (203.25,54.78) .. (201.7,51.12) -- cycle ;
\draw  [draw opacity=0][fill={rgb, 255:red, 0; green, 0; blue, 0 }  ,fill opacity=1 ] (206.14,49.46) .. controls (205.64,48.29) and (206.19,46.94) .. (207.36,46.44) .. controls (208.53,45.94) and (209.88,46.49) .. (210.37,47.66) .. controls (210.87,48.83) and (210.33,50.18) .. (209.16,50.68) .. controls (207.99,51.17) and (206.64,50.63) .. (206.14,49.46) -- cycle ;

\draw    (197.35,38.05) .. controls (219.57,57.14) and (222.57,69.14) .. (221.57,100.14) ;
\draw  [draw opacity=0][fill={rgb, 255:red, 255; green, 255; blue, 255 }  ,fill opacity=1 ] (216.37,81.97) .. controls (213.37,79.36) and (213.05,74.81) .. (215.66,71.81) .. controls (218.28,68.81) and (222.82,68.49) .. (225.82,71.11) .. controls (228.82,73.72) and (229.14,78.26) .. (226.53,81.26) .. controls (223.92,84.26) and (219.37,84.58) .. (216.37,81.97) -- cycle ;
\draw  [draw opacity=0][fill={rgb, 255:red, 0; green, 0; blue, 0 }  ,fill opacity=1 ] (219.64,78.53) .. controls (218.68,77.7) and (218.58,76.24) .. (219.41,75.29) .. controls (220.25,74.33) and (221.7,74.23) .. (222.66,75.06) .. controls (223.62,75.9) and (223.72,77.35) .. (222.88,78.31) .. controls (222.05,79.27) and (220.6,79.37) .. (219.64,78.53) -- cycle ;

\draw  [draw opacity=0][fill={rgb, 255:red, 255; green, 255; blue, 255 }  ,fill opacity=1 ] (192.73,81.38) .. controls (189.73,78.77) and (189.41,74.22) .. (192.02,71.22) .. controls (194.63,68.22) and (199.18,67.91) .. (202.18,70.52) .. controls (205.18,73.13) and (205.5,77.68) .. (202.89,80.68) .. controls (200.27,83.68) and (195.73,83.99) .. (192.73,81.38) -- cycle ;
\draw  [draw opacity=0][fill={rgb, 255:red, 0; green, 0; blue, 0 }  ,fill opacity=1 ] (196,77.95) .. controls (195.04,77.11) and (194.94,75.66) .. (195.77,74.7) .. controls (196.61,73.74) and (198.06,73.64) .. (199.02,74.48) .. controls (199.98,75.31) and (200.08,76.76) .. (199.24,77.72) .. controls (198.41,78.68) and (196.96,78.78) .. (196,77.95) -- cycle ;

\draw    (175.68,68.13) .. controls (196.57,81.14) and (222.57,81.14) .. (242.38,69.18) ;

\end{tikzpicture}
    \caption{R-III move}
    \vspace{3em}
    \includegraphics[scale=0.4]{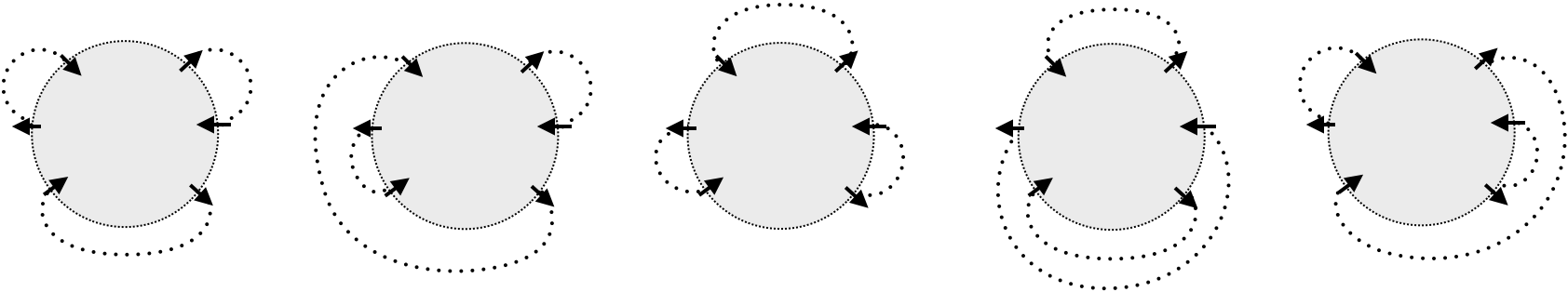}
    \caption{Possible connections of the endpoints}
    \label{fig:RM3-strands}
    \vspace{3em}
    \begin{subfigure}{\textwidth}
        \centering
        \input{tikz/RM3_pattern1}
        \caption{}
        \label{fig:RM3-pattern1}
    \end{subfigure}
    \begin{subfigure}{\textwidth}
        \centering
        \vspace{1.5em}
        \input{tikz/RM3_pattern2}
        \caption{}
        \label{fig:RM3-pattern2}
    \end{subfigure}
    \begin{subfigure}{\textwidth}
        \centering
        \vspace{1.5em}
        \input{tikz/RM3_pattern3}
        \caption{}
        \label{fig:RM3-pattern3}
    \end{subfigure}
    \vspace{1em}
    \caption{Possible patterns of the $\ca$-, $\cb$-cycles of $D, D'$ and their correspondences under $\rho$ (modulo boundary).}
    \label{fig:RM3}
\end{figure}


\begin{remark}
    Plamenevskaya's invariant $\psi(L)$ \cite{Plamenevskaya:2006} of a transverse link $L$ is represented by $\ca(D)$ with $h = t = 0$. The proof of \Cref{lem:adj-rho} also applies when $c = 0$ and $j \geq 0$, so in particular it follows that $\ca(D)$ is strictly invariant under transverse Markov moves. Thus Plamenevskaya's invariant can be refined so that there is no sign indeterminacy. The same applies to the filtered version defined by Lipshitz-Ng-Sarkar in \cite{LNS:2015}.
\end{remark}

\subsection{Adjustments for Morse moves} \label{subsec:adj-morse-moves}

\begin{proposition} \label{lem:adj-phi}
    Suppose two link diagrams $D, D'$ are related by one of the Morse moves. Let $S$ be the corresponding elementary cobordism. The sign of $\BN(S)$ can be adjusted so that for any $(R, h, t)$ satisfying $c = \sqrt{h^2 + 4t} \in R^\times$, the corresponding isomorphism 
    \[
        \phi_S: H_{h, t}(D; R) \rightarrow H_{h, t}(D'; R)
    \]
    maps the $\ca$-, $\cb$-classes
    \begin{align*}
        [\ca(D)] &\xmapsto{\phi_S} \hspace{1.5em} c^j [\ca(D')] + \cdots, \\
        [\cb(D)] &\xmapsto{\phi_S} (-c)^j [\cb(D')] + \cdots
    \end{align*}
    where $j \in \{\pm 1\}$ is given by 
    \[
        j = \frac{-\delta r(D, D') - \chi(S)}{2}.
    \]
    Moreover the $(\cdots)$ terms are only present for a cup move.
\end{proposition}

\begin{proof}
    As in the proof of \Cref{lem:adj-rho}, we take one $(R, h, t)$ satisfying $2 \neq 0 \in R$ and $c = \sqrt{h^2 + 4t} \in R^\times$. Recall that the homomorphism
    \[
        \phi_S: H_{h, t}(D; R) \rightarrow H_{h, t}(D'; R)
    \]
    is given by the operations of the Frobenius algebra $A$, namely, the unit $\iota$ for a cup, the counit $\epsilon$ for a cap, and the mix of the multiplication $m$ and the comultiplication $\Delta$ for a saddle. First for a cup, we have
    \[
        [\ca(D)] \ \xmapsto{\phi_S} \ 
            [\ca(D) \otimes 1] 
            = c^{-1}([\ca(D) \otimes \a] - [\ca(D) \otimes \b])
    \]
    and either one of $[\ca(D) \otimes \a], [\ca(D) \otimes \b]$ is equal to $[\ca(D')]$. For a cap, from $\epsilon(\a) = \epsilon(\b) = 1$, we have
    \[
        [\ca(D)] \xmapsto{\phi_S} [\ca(D')].
    \]
    Finally for a saddle, 
    \[
        [\ca(D)] \xmapsto{\phi_S} \begin{cases}
            \pm c [\ca(D')] \\
            \hspace{1.2em} [\ca(D')]
        \end{cases}
    \]
    depending on whether the saddle merges or splits the Seifert circles of $D$. Thus in either case, the image $\phi_S([\ca(D)])$ contains $[\ca(D')]$ in one of its terms. Redefine $\BN(S)$ by multiplying $\pm 1$ so that the $[\ca(D')]$ term in $\phi_S[\ca(D)]$ is positive. The equation for $j$ can be checked easily.
\end{proof}

\subsection{Adjusted cobordism maps}

\begin{proposition} \label{lem:cob-no-closed-comp}
    Suppose $S$ is a generic cobordism between links $L, L'$ with no closed components. Let $D, D'$ be the diagrams of $L, L'$ respectively. With the sign of $\BN(S)$ adjusted as in \Cref{lem:adj-rho,lem:adj-phi}, for any $(R, h, t)$ satisfying $c = \sqrt{h^2 + 4t} \in R^\times$, the cobordism map
    \[
        \phi_S: H_{h, t}(D; R) \rightarrow H_{h, t}(D'; R)
    \]
    maps the $\ca$-, $\cb$-classes
    \begin{align*}
        [\ca(D)] &\xmapsto{\phi_S} \hspace{1.5em} c^j [\ca(D')] + \cdots, \\
        [\cb(D)] &\xmapsto{\phi_S} (-c)^j [\cb(D')] + \cdots
    \end{align*}
    where
    \[
        j = \frac{\delta w(D, D') - \delta r(D, D') - \chi(S)}{2}.
    \]
    Moreover, if every component of $S$ has a boundary in $L$, then the $(\cdots)$ terms vanish.
\end{proposition}

\begin{proof}
    We have proved that the result is true when $S$ is an elementary cobordism. The case when $S$ is a composition of elementary cobordisms proceeds as in the proof of \cite[Proposition 3.17]{Sano:2020}.
\end{proof}
    \section{Proofs of the main theorems}
\label{sec:proof}

Now we are ready to prove the main theorems stated in \Cref{sec:intro}. 

\begin{proof}[Proof of \Cref{thm:main}]
    Suppose $S$ and $S'$ are isotopic link cobordisms between link diagrams $D$ and $D'$ such that the pair $(S, S')$ corresponds to one of the fifteen \textit{movie moves} (or those reverses) of Carter and Saito \cite{CS:1993}. We have corresponding morphisms
    \[
        \BN(S), \BN(S'): \BN(D) \rightarrow \BN(D'),
    \]
    and from \cite[Theorem 4]{BarNatan:2004} we know that $\BN(S) \htpy \epsilon \BN(S')$ for some $\epsilon \in \{\pm1\}$. Our aim is to prove that $\epsilon = 1$, after adjusting the signs of $\BN$ as in \Cref{lem:adj-rho,lem:adj-phi}. Now take $(R, h, t) = (\ZZ, 1, 0)$. Obviously $c = 1$ is invertible and the results of the previous section are applicable.  Applying the functor $H \circ \mathcal{F}_{1, 0}$ gives cobordism maps between the corresponding Khovanov type homologies
    \[
        \phi_S, \phi_{S'}: H_{1, 0}(D; \ZZ) \rightarrow H_{1, 0}(D'; \ZZ)
    \]
    and we have $\phi_S = \epsilon \phi_{S'}$. From the description of the movie moves, one sees that both $S$ and $S'$ have no closed components. Thus from \Cref{lem:cob-no-closed-comp}, the images of $[\ca(D)] \in H_{h, t}(D; R)$ under $\phi_S$ and $\phi_{S'}$ are both of the form $[\ca(D')] + \cdots$. Since the canonical classes form bases of the homologies, we must have $\epsilon = 1$.
\end{proof}

\begin{proof}[Proof of \Cref{thm:closed-surf-inv}]
    Consider ring extensions
    \[
        R_0 = \ZZ[h, t]
        \ \hookrightarrow \ 
        R_1 = R_0[u, v]
        \ \hookrightarrow \ 
        R_2 = R_1[c^{-1}]
    \]
    where $u, v$ are roots of $X^2 - hX - t \in R_0[X]$ and $c = v - u$. Regarding $S$ as a cobordism between empty links, we obtain the following commutative diagram.
    \begin{figure}[H]
        \centering
\begin{tikzcd}
R_0 = \mathit{Kh}(\emptyset; R_0) \arrow[d, hook] \arrow[rr, "\phi_S"] &  & \mathit{Kh}(\emptyset; R_0) = R_0 \arrow[d, hook] \\
R_2 = \mathit{Kh}(\emptyset; R_2) \arrow[rr, "\phi_S"]                     &  & \mathit{Kh}(\emptyset; R_2) = R_2                  
\end{tikzcd}
    \end{figure}
    \noindent
    Thus it suffices to prove the equation for $R = R_2$. Isotope $S$ so that $S$ decomposes as $\text{(cup)} \cup S' \cup \text{(cap)}$, and that the projection of the two boundary circles of $S'$ are both crossingless unknot diagrams $U, U'$. From \Cref{lem:cob-no-closed-comp}, 
    \begin{align*}
        [\ca(U)] &\xmapsto{\phi_{S'}} \hspace{1.5em} c^j [\ca(U')], \\
        [\cb(U)] &\xmapsto{\phi_{S'}} (-c)^j [\cb(U')]
    \end{align*}
    where $j = -\chi(S') / 2 = g(S)$. The cup map sends $1 \in R$ to $c^{-1}([\ca(U)] - [\cb(U)])$, and the cap map sends both $[\ca(U')], [\cb(U')]$ to $1 \in R$. Thus $\phi_S$ maps
    \begin{align*}
        1 &\xmapsto{\phi_\cup} c^{-1}([\ca(U)] - [\cb(U)]) \\
          &\xmapsto{\phi_{S'}} c^{j-1} ([\ca(U')] + (-1)^{j+1} [\cb(U')]) \\
          &\xmapsto{\phi_\cap} c^{j-1} (1 + (-1)^{j+1}).
    \end{align*}
    Substituting $h^2 + 4t$ for $c^2$ gives the desired equation.
\end{proof}

\begin{remark}
The value of the Khovanov-Jacobsson number $n_S = |H_{0, 0}(S)|$ is determined by Rasmussen in \cite{Rasmussen:2005}, and independently by Tanaka \cite{Tanaka:2006} by the specialization $n_S = |H_{0, t}(S)|_{t = 0}$. Caprau \cite{Cap:2009} defined an invariant $\mathit{Inv}(S) \in \ZZ[i][h, t]$ using her universal $\mathfrak{sl}(2)$ cohomology.
\end{remark}


    \section*{Acknowledgements}

The author is grateful to his supervisor Mikio Furuta for the support. He thanks Tomohiro Asano and Kouki Sato for helpful suggestions, and the anonymous referee for detailed corrections and suggestions. He thanks members of his \textit{academist fanclub}\footnote{\url{https://taketo1024.jp/supporters}} for the support. This work was supported by JSPS KAKENHI Grant Number 20J15094. 
    \printbibliography

\end{document}